\documentclass[11pt,leqno]{amsart}
\usepackage{enumitem}
\usepackage[utf8]{inputenc} 
\usepackage[T1]{fontenc}    
\usepackage{url}            
\usepackage{booktabs}       
\usepackage{amsfonts}
\usepackage{mathrsfs}         
\usepackage{amsmath,amssymb,bbm}
\usepackage{amsthm}
\usepackage{nicefrac}       
\usepackage{microtype}      
\usepackage[T1]{fontenc}
\usepackage{lmodern}
\usepackage[all]{xy}

\usepackage[colorlinks=true,pagebackref=true]{hyperref} 

\newtheorem{theorem}{Theorem}[section]

\newtheorem{proposition}[theorem]{Proposition}
\newtheorem{lemma}[theorem]{Lemma}
\newtheorem{corollary}[theorem]{Corollary}

\newtheorem{remark}{Remark}

\newcommand{\MW}{Milnor-Witt\ }
\newcommand{\rMW}{\mathrm{MW}}
\newcommand{\KMW}{\mathrm{K}^\mathrm{MW}}
\newcommand{\bb}[1]{\mathbb{#1}}

\newcommand{\wt}[1]{\widetilde{#1}}
\newcommand{\af}{\mathbb{A}_K}

\newcommand{\bZ}{\mathbb{Z}}

\newcommand{\mrm}[1]{\mathrm{#1}}
\newcommand{\inv}{^{-1}}
\newcommand{\Gm}{\bb{G}_m}

\newcommand{\DMt}{\widetilde{\mathrm{DM}}(K)}

\newcommand{\Mt}{\wt{\mrm{M}}}
\newcommand{\HH}{\mathrm{H}_{\rMW}}
\newcommand{\sBra}[1]{\left[#1\right]}
\newcommand{\aBra}[1]{\left<#1\right>}

\newcommand{\rHom}{\mathrm{Hom}}

\newcommand{\xr}[1]{\xrightarrow{#1}}

\newcommand{\one}{\mathbbm 1}
	
\title[MW-motivic cohomology and hyperplane arrangements]{Milnor-Witt motivic cohomology of complements of hyperplane arrangements}

\author{Keyao Peng}
\address{Institut Fourier, Universit\'e Grenoble-Alpes, CS40700, 38058 Grenoble Cedex 9, France}

\date{}

\begin{document}
\maketitle

\begin{abstract}
In this paper, we compute the (total) Milnor-Witt motivic cohomology of the complement of a hyperplane arrangement in an affine space as an algebra with given generators and relations. We also obtain some corollaries by realization to classical cohomology.
\end{abstract}


\section{Introduction}

Let $K$ be a perfect field of characteristic different from $2$, and let $ U \subset \af^N $ be the complement of a finite union of hyperplanes. For $K=\bb{R}$ the cohomology ring $\mrm{H}^*_{\mrm{sing}}(U(\bb{R}),\bZ)$ is just the direct sum of $\bZ$ corresponding to each regions(connected components), and those regions form a poset. In the special case when the hyperplanes arise from a root system, the resulting poset is the corresponding Weyl group with the weak Bruhat order. In general, the poset of regions is ranked by the number of separating hyperplanes and its Möbius function has been computed \cite{10.2307/1999150}.

For any essentially smooth scheme $X$ over $K$ and any integers $p,q\in\bZ$, one can define the Milnor-Witt motivic cohomology groups $\HH^{p,q}(X,\bZ)$ introduced in \cite{MW2020}. There are homomorphisms (functorial in $X$) for any $p,q\in \bZ$
\[
\HH^{p,q}(X,\bZ)\to \mathrm{H}_{\mathrm{M}}^{p,q}(X,\bZ)
\]
where the right-hand side denotes the ordinary motivic cohomology defined by Voevodsky. 

As illustrated by the list of properties in the following section, the Milnor-Witt motivic cohomology groups behave in a fashion similar to ordinary motivic cohomology groups, but there are central differences (for instance, there are no reasonable Chern classes). 

In this paper, we compute the total Milnor-Witt cohomology ring of the complement of a hyperplane arrangement in affine spaces $\HH(U)$ using methods very similar to \cite{math/0601737}, with some necessary modifications. To state our main result, we first recall a few facts.

Let $R$ be a commutative ring. The Milnor-Witt $K$-theory of $R$ is defined to be the graded algebra freely generated by elements of degree $1$ of the form $[a], a\in R^\times$ and an element $\eta$ in degree $-1$, subject to the relations
\begin{enumerate}
\item $[a][1-a]=0$ for any $a$ such that $a, 1-a\in R^\times \setminus \{1\}$.
\item $[ab]=[a]+[b]+\eta[a][b]$ for any $a,b\in R^\times$.
\item $\eta[a]=[a]\eta$ for any $a\in R^\times$.
\item $\eta(2+\eta[-1])=0$.
\end{enumerate}
It defines a presheaf on the category of schemes over a perfect field $K$ via $X\mapsto \mathrm{K}_*^{\rMW}(\mathcal{O}(X))$. On the other hand, one can also consider the Milnor-Witt motivic cohomology (bigraded) presheaf
\[
X\mapsto \HH(X):= \oplus _{p,q}\mathrm{H}^{p,q}_{\rMW}(X,\bb{Z})
\]
By \cite[Theorem 4.2.2]{Deglise16}, there is a morphism of presheaves
\[
s:\bigoplus_{n\in\bb{Z}}\mathrm{K}_n^{\rMW}(-)\to \bigoplus_{n\in\bb{Z}}\mathrm{H}_{\rMW}^{n,n}(-,\bb{Z})\subset \HH(X)
\]
which specializes to the above isomorphism if $X=\mathrm{Spec}(F)$ where $F$ is a finitely generated field extension of $K$ \cite{Calmes17a}. 

\begin{theorem}\label{thm:mainthm}
	Let $ K $ be a perfect field of characteristic different from $2$ and let $U\subset \bb{A}^N_K$ be the complement of a finite union of hyperplanes. There is an isomorphism of $\HH(K)$-algebras
	\[ \HH(K)\{\Gm(U)\}/J_U \cong 
	\HH(U)\]
	 defined by mapping $ (f)  \in \Gm(U) $ to the class $[f]$ in $\HH^{1,1}(U,\bb{Z})$ corresponding to $f$ under $s$. Here, $\HH(K)\{\Gm(U)\}$ is the free (associative) graded $\HH(K)$-algebra generated by $\Gm(U)$ in degree $(1,1)$ and $J_U$ is the ideal generated by the following elements:
	 \begin{enumerate}[label=(\arabic*)]
	 	\item $(f)-[f]$, if $ f \in K^{\times}\subset \Gm(U) $,
	 	\item $ (f)+(g)+\eta (f)(g)- (fg) $, if $ f,g\in \Gm(U)$,
	 	\item $(f_1)(f_2)\cdots(f_t)$, for any $f_1,\ldots,f_t\in \Gm(U)$ such that $\sum_{i=1}^t f_i =1 $,
	 	\item $(f)^2-[-1](f)$, if $ f \in\Gm(U)$. 
	 \end{enumerate}
\end{theorem}
As indicated above, this theorem and its proof are inspired from the computation of the (ordinary) motivic cohomology of $U$ in \cite{math/0601737}. We can recover the main theorem \cite[Theorem 3.5]{math/0601737} of the motivic cohomology case by taking $ \eta = 0 $.
As a corollary, we obtain the following result.

\begin{corollary}
Let $U\subset \af^N$ be the complement of a finite union of hyperplanes. The isomorphism of Theorem \ref{thm:mainthm} induces an isomorphism
\[
\bigoplus_{n\in\bb{Z}}\mathrm{K}_n^{\rMW}(K)\{\Gm(U)\}/J_U\to \bigoplus_{n\in\bb{Z}}\mathrm{H}_{\rMW}^{n,n}(U,\bb{Z}).
\]
\end{corollary}

We do not know if the left-hand side coincides with $\mathrm{K}_*^{\rMW}(U)$. At the end of the paper, we spend a few lines on the real realization homomorphism
\[
\HH(U,\bZ)\to  \mathrm{H}^{*}_{\mathrm{sing}}(U(\mathbb{R}),\bZ)
\]
when $U$ is over $K=\mathbb{R}$. We prove in particular that both sides have essentially the same generators, and that the map is surjective.

\subsection*{Conventions}

The base field $K$ is assumed to be perfect and of characteristic not 2. For a scheme $X$ over $K$, we write $\HH(X)$ for the total MW-motivic cohomology ring  $\bigoplus_{p,q\in\bb{Z}} \HH^{p,q}(X,\bb{Z})$.

For all $f\in \Gm(U)$, we use $(f)$ to indicate the corresponding generator in the corresponding free algebras (e.g. $\mathrm{K}_n^{\rMW}(K)\{\Gm(U)\}$) and $[f]$ to indicate the corresponding elements in the cohomology groups (e.g. $\mathrm{H}_{\rMW}^{1,1}(U,\bb{Z})$).

\section{\MW Motivic Cohomology}\label{sec:MWmotivic}
In this section, we define \MW motivic cohomology and state some properties that will be used in the proof of Theorem \ref{thm:mainthm}. We start with the (big) category of motives $\widetilde{\mathrm{DM}}(K):=\widetilde{\mathrm{DM}}_{\mathrm{Nis}}(K,\bb{Z})$ defined in \cite[Definition 3.3.2]{Deglise16} and the functor 
\[
\Mt : \mathrm{Sm}/K\to \DMt.
\]
The category $\DMt$ is symmetric monoidal \cite[Proposition 3.3.4]{Deglise16} with unit $\one=\Mt(\mathrm{Spec}(K))$. For any integers $p,q\in \bb{Z}$, we obtain MW-motivic cohomology groups 
\[
\HH^{p,q}(X,\bZ):=\rHom_{\DMt}(\Mt(X),\one(q)[p]).
\]
By \cite[Proposition 4.1.2]{Deglise16}, motivic cohomology groups can be computed as the Zariski hypercohomology groups of explicit complexes of sheaves.
%

We will make use of the following property of $\DMt$. First, we note that $\DMt$ is also a triangulated category.
\begin{proposition}[Gysin Triangle]
	\label{Gysin}
	Let $ X $ be a smooth K-scheme, let $ Z\subset X $ be a smooth closed subscheme of codimension $c$ and let  $U=X\setminus Z$. Suppose that the normal cone $ N_X Z $ admits a trivialization $\phi:N_X Z\cong Z\times \mathbb{A}^c$. Then, there is a Gysin triangle
	\[ \Mt(U)\xr{}\Mt(X)\xr{}\Mt(Z)(c)[2c]\xr{+1} \]
	where the last two arrows depend on the choice of $\phi$.
\end{proposition}
\begin{proof}
We have an adjunction of triangulated categories
\[
\mathrm{SH}(K)\leftrightarrows \DMt
\]
obtained by combining the adjunction of \cite[\S 4.1]{Deglise17} and the classical Dold-Kan correspondence (e.g. \cite[5.3.35]{Cisinski09b}). Here, $\mathrm{SH}(K)$ is the stable homotopy category of smooth schemes over $K$. The functor $\mathrm{SH}(K)\to \DMt$ being exact, the statement follows for instance from \cite[Chapter 3, Theorem 2.23]{Morel99}.
\end{proof}

Futhermore, the Milnor-Witt motivic cohomology groups satisfy most of the formal properties of ordinary motivic cohomology and were computed in a couple of situations:
\begin{enumerate}
\item If $q\leq 1$, there are canonical isomorphisms 
\[
\HH^{p,q}(X,\bZ)\cong \mathrm{H}^{p-q}_{\mathrm{Nis}}(X,\mathbf{K}_q^{\mathrm{MW}})\cong \mathrm{H}^{p-q}_{\mathrm{Zar}}(X,\mathbf{K}_q^{\mathrm{MW}})
\]
where $\mathbf{K}_q^{\mathrm{MW}}$ is the unramified Milnor-Witt $K$-theory sheaf (in weight $q$) introduced in \cite{Morel08}.
\item If $L/K$ is a finitely generated field extension there are isomorphisms $\HH^{n,n}(L,\bZ)\cong \mathrm{K}^{\mathrm{MW}}_n(L)$ fitting in a commutative diagram for any $n\in\bZ$
\[
\xymatrix{\HH^{n,n}(L,\bZ)\ar[r]^-{\sim}\ar[d] & \mathrm{K}^{\mathrm{MW}}_n(L)\ar[d] \\
 \mathrm{H}_{\mathrm{M}}^{n,n}(L,\bZ)\ar[r]^-{\sim} & \mathrm{K}^{\mathrm{M}}_n(L)}
\]
where $\mathrm{K}^{\mathrm{M}}_n(L)$ is the ($n$-th) Milnor $K$-theory group of $L$, the bottom horizontal map is the isomorphism of Suslin-Nesterenko-Totaro, and the right-hand vertical map is the natural homomorphism from Milnor-Witt $K$-theory to Milnor $K$-theory. This result has the following consequence: The Milnor-Witt motivic cohomology groups are computed via an explicit complex of Nisnevich sheaves $\tilde\bZ(q)$ for any integer $q\in \bZ$. The above result shows that there is a morphism of complexes of sheaves
\[
\tilde\bZ(q)\to  \mathrm{K}^{\mathrm{MW}}_q[-q]
\]
where the right-hand side is the complex whose only non-trivial sheaf is $ \mathrm{K}^{\mathrm{MW}}_q$ in degree $-q$. For any essentially smooth scheme $X$ over $K$, this yields group homomorphisms
\[
\HH^{p,q}(X,\bZ)\to \mathrm{H}^{p-q}(X,\mathbf{K}^{\mathrm{MW}}_q)
\]
which are compatible with the ring structure on both sides. In the particular case $p=2n, q=n$ for some $n\in\bZ$, we obtain isomorphisms (functorial in $X$)
\[
\HH^{2n,n}(X,\bZ)\xrightarrow{\sim} \widetilde{\mathrm{CH}}^n(X)
\]
where the right-hand term is the $n$-th Chow-Witt group of $X$ (defined in \cite{Barge00} and \cite{Fasel08a}). Again, these isomorphisms fit into commutative diagrams
\[
\xymatrix{\HH^{2n,n}(X,\bZ)\ar[r]^-{\sim}\ar[d] & \widetilde{\mathrm{CH}}^n(X)\ar[d] \\
 \mathrm{H}_{\mathrm{M}}^{2n,n}(X,\bZ)\ar[r]^-{\sim} & \mathrm{CH}^n(X)}
\]
where the right-hand vertical homomorphism is the natural map from Chow-Witt groups to Chow groups.
\item The total Milnor-Witt motivic cohomology has Borel classes for symplectic bundles \cite{Yang17} but in general the projective bundle theorem fails \cite{Yang20}.
\item If $X$ is a smooth scheme over $\mathbb{R}$, there are two interesting realization maps. On the one hand, one may consider the composite
\[
\HH^{p,q}(X,\bZ)\to \mathrm{H}_{\mathrm{M}}^{p,q}(X,\bZ)\to \mathrm{H}^p_{\mathrm{sing}}(X(\mathbb{C}),\bZ)
\]
where the right-hand map is the complex realization map. On the other hand, one may also consider the following composite
\[
\HH^{p,q}(X,\bZ)\to \mathrm{H}^{p-q}(X,\mathbf{K}^{\mathrm{MW}}_q)\to \mathrm{H}^{p-q}(X,\mathbf{I}^q)\to \mathrm{H}^{p-q}_{\mathrm{sing}}(X(\mathbb{R}),\bZ)
\]
where $\mathbf{I}^q$ is the unramified sheaf associated to the $q$-th power of the fundamental ideal in the Witt ring, $\mathbf{K}^{\mathrm{MW}}_q\to \mathbf{I}^q$ is the canonical projection and $\mathrm{H}^{p-q}(X,\mathbf{I}^q)\to \mathrm{H}^{p-q}_{\mathrm{sing}}(X(\mathbb{R}),\bZ)$ is Jacobson's signature map \cite{Jacobson17}.

We note here that these two realization maps show that Milnor-Witt motivic cohomology is in some sense the analogue of both the singular cohomology of the complex and the real points of $X$.
\end{enumerate}

\section{Basic structure of the cohomology ring}
Let $V$ be an affine space, i.e. $V\cong \af^N$ for some $N\in\bb{N}$. We consider finite families $I$ of hyperplanes in $V$ (that we suppose are distinct). We denote by $\vert I\vert$ the cardinality of $I$ and set $U^V_I:=V\setminus (\cup_{Y\in I} Y)$ and simply write $U^N_I$ when $V=\af^N$. For any hyperplane $Y$, we put $I_{Y}:=\{Y_i \cap Y | Y_i\in I, Y_i\neq Y\}$. 

\begin{proposition}\label{decomp}
Let $V$ and $I$ be as above. We have
\[
\Mt(U^V_I)\cong \oplus_{j\in J}\one(n_j)[n_j]
\]
for some set $J$ and integers $n_j \geq 0$.
\end{proposition}

\begin{proof}
We proceed by induction on the dimension $N$ of $V$ and $\vert I\vert$. If $\vert I\vert=0$, then $\Mt(U^V_I)=\Mt(V)\cong \one$ and we are done. Let then $\vert I\vert\geq 1$ and $Y\in I$. The Gysin triangle reads as 
\begin{equation}
	\label{gysin1}
	\Mt(U^V_I)\to\Mt(U^V_{I-\{Y\}})\xr{\phi}\Mt(U^Y_{I_Y})(1)[2]\xr{+1} .
	\end{equation}
If $\phi = 0 $, then the triangle is split and consequently we obtain an isomorphism
\begin{equation}
	\label{split}
\Mt(U^V_I)\cong \Mt(U^V_{I-\{Y\}}) \oplus \Mt(U^{Y}_{I_Y})(1)[1].
	\end{equation}
Since $\vert I-\{Y\}\vert  < \vert I\vert $ and $\mathrm{dim}(Y)=\mathrm{dim}(V)-1$ we conclude by induction that the right-hand side has the correct form. We are then reduced to show that $\phi = 0 $.

By induction,
\begin{align*}
	\phi&\in \rHom_{\DMt} (\Mt(U^V_{I-\{Y\}}),\Mt(U^{Y}_{I_Y})(1)[2]) \\
 &\cong \bigoplus_{j,k}\rHom_{\DMt}(\one(n_j)[n_j],\one(m_k)[m_k+1])
\end{align*}

for some integers $n_j,m_k\geq 0$, and it suffices to prove that \\ $\rHom_{\DMt}(\one,\one(m)[m+1] )=0$ for any $m\in\bb{Z}$ to conclude. Now,
\[
\rHom_{\DMt}(\one,\one(m)[m+1] )=\HH^{m+1,m}(K,\bb{Z})
\] 
and the latter is trivial by \cite[Proposition 4.1.2]{Deglise16} and \cite[proof of Theorem 4.2.4]{Deglise16}.
\end{proof}

As an immediate corollary, we obtain the following result.

\begin{corollary}\label{cor:free}
The motivic cohomology $\HH(U^V_I)$ is a finitely generated free $\HH(K)$-module.
\end{corollary} 

To obtain more precise results, we now study the Gysin (split) triangle \eqref{gysin1} in more detail. We can rewrite it as
\[
\Mt(U^Y_{I_Y})(1)[1]\xr{\beta^Y}\Mt(U^V_I)\xr{\alpha^Y}\Mt(U^V_{I-\{Y\}})\xr{0}
\]
and therefore we obtain the following  short (split) exact sequence in which the morphisms are induced by the first two morphisms in the triangle
\begin{equation}\label{short}
0\xr{}\bigoplus_{p,q}\HH^{p,q}(U^V_{I-\{Y\}},\bb{Z})
\xr{\alpha^Y_*}\bigoplus_{p,q}\HH^{p,q}(U^V_{I},\bZ)
\xr{\beta^Y_*}\bigoplus_{p,q}\HH^{p-1,q-1}(U^Y_{I_Y},\bZ)
\xr{}0.
\end{equation}

The inclusion $Y\subset V$ yields a morphism $U^Y_{I_Y} \to U^V_{I-\{Y\}}$ and therefore a morphism $\iota:\Mt(U^Y_{I_Y})\to \Mt(U^V_{I-\{Y\}})$. The global section $f$ of $V$ corresponding to the equation of $Y$ becomes invertible in $U^V_I$ and therefore yields a morphism $[f]:\Mt(U^V_I)\to \one(1)[1]$ corresponding to the class $[f]\in \HH^{1,1}(U^V_I,\bZ)$ given by the morphism 
\[
s:\bigoplus_{n\in\bb{Z}}\mathrm{K}_n^{\rMW}(-)\to \bigoplus_{n\in\bb{Z}}\mathrm{H}_{\rMW}^{n,n}(-,\bb{Z}).
\]

\begin{lemma}\label{lem:computation}
The following diagram commutes
\[
\xymatrix{\Mt(U^Y_{I_Y})(1)[1]\ar[d]_-{\beta^Y}\ar[r]^-{\iota(1)[1]} & \Mt(U^V_{I-\{Y\}})(1)[1] \\
\Mt(U_I^V)\ar[r]_-{\Delta} & \Mt(U_I^V)\otimes \Mt(U_I^V)\ar[u]_-{\alpha^Y\otimes[f]}.}
\]
\end{lemma}

\begin{proof}
The commutative diagram of schemes
\[
\xymatrix{U_I^V\ar[r]\ar[d] & U^V_{I-\{Y\}}\ar[d]^-{(\mathrm{Id},f)}  \\
U^V_{I-\{Y\}}\times \Gm\ar[r] & U^V_{I-\{Y\}}\times \af^1     }
\]
yields a morphism of Gysin triangles and thus a commutative diagram
\[
\xymatrix{\Mt(U^Y_{I_Y})(1)[1]\ar[r]^-{\beta^Y}\ar[d]_-{\iota(1)[1]} & \Mt(U_I^V)\ar[r]^-{\alpha^Y}\ar[d] & \Mt(U^V_{I-\{Y\}})\ar[r]\ar[d] & \ldots\\
 \Mt(U^V_{I-\{Y\}})(1)[1]\ar[r]\ar@{=}[rd] &  \Mt(U^V_{I-\{Y\}}\times \Gm)\ar[r]\ar[d] &  \Mt(U^V_{I-\{Y\}}\times \af^1)\ar[r] & \ldots \\
 &  \Mt(U^V_{I-\{Y\}})(1)[1] & & }
\]
in which the map $\Mt(U^V_{I-\{Y\}}\times \Gm)\to  \Mt(U^V_{I-\{Y\}})(1)[1]$ is just the projection. We conclude by observing that the middle vertical composite is just $(\alpha^Y\otimes[f])\circ \Delta$.
\end{proof}

We may now prove the main result of this section. 

\begin{proposition}\label{prop:generation}
The cohomology ring $\HH(U)$ is generated by the classes of units in $U$ as an $\HH(K)$ algebra. In particular, the homomorphism 
\[
s:\bigoplus_{n\in\bb{Z}}\mathrm{K}_n^{\rMW}(U)\to \bigoplus_{n\in\bb{Z}}\mathrm{H}_{\rMW}^{n,n}(U,\bb{Z})
\]
is surjective.
\end{proposition}

\begin{proof}
We still prove the result by induction on $\vert I\vert$ and the dimension of $V$, the case $\vert I\vert =0$ being obvious. Suppose then that the result holds for $U^Y_{I_Y}$ and  $U^V_{I-\{Y\}}$ and consider the split sequence \eqref{short}. For any $x\in \HH(U)=\HH(U^V_I)$, we have that $\beta^Y_*(x)\in \HH(U^Y_{I_Y})$ is in the subalgebra generated by $\{[f]\vert f\in \Gm (U^Y_{I_Y})\}$ and $\eta$. For any $f_1,\ldots,f_n\in \Gm(U^V_{I-\{Y\}})$, Lemma \ref{lem:computation} yields
\[
\beta^Y_*([(f_1)_{\vert U^V_I}]\cdots[(f_n)_{\vert U^V_I}]\cdot [t])=[(f_1)_{\vert U^Y_{I_Y}}]\cdots[(f_n)_{\vert U^Y_{I_Y}}].
\]
The map $\Gm(U^V_{I-\{Y\}})\to \Gm(U^Y_{I_Y})$ being surjective, it follows that there exists $x^\prime\in \HH(U^V_I)$ in the subalgebra generated by units such that $\beta^Y_*(x-x^\prime)=0$. Thus $x-x^\prime=\alpha_*(y)$ for some $y\in \HH(U^V_{I-\{Y\}})$ and the result now follows from the fact that $\alpha_*$ is just induced by the inclusion $U^V_I\subset U^V_{I-\{Y\}}$.
\end{proof}

\section{Relations in the cohomology ring}

The purpose of this section is to prove that the relations of Theorem \ref{thm:mainthm} hold in $\HH(U)$. The first two relations are obviously satisfied since the homomorphism is induced by the ring homomorphism
\[
s:\bigoplus_{n\in\bb{Z}}\mathrm{K}_n^{\rMW}(U)\to \bigoplus_{n\in\bb{Z}}\mathrm{H}_{\rMW}^{n,n}(U,\bb{Z}).
\]
Recall now that the last two relations are

\begin{itemize}
	\item[3.] $[f_1][f_2]\cdots[f_t]$, if $f_i \in \Gm(U)$ for any $i$ and $\sum_{i=1}^t f_i =1 $,
	\item[4.] $[f]^2-[-1][f]$, if $ f \in\Gm(U)$. 
\end{itemize}
We will prove that they are equal to $ 0 $ in $\HH(U)$. Actually, it will be more convenient to work with the following relations
\begin{itemize}
\item[$3^\prime$.] $R(f_0,\ldots,f_t)$ defined by 
\[\sum_{i=0}^t \epsilon^{t+i}\sBra{f_0}\ldots \widehat{\sBra{f_i}} \ldots\sBra{f_t} +
	\sum_{0\leq i_0<\ldots<i_k\leq t} (-1)^k\sBra{-1}^k\sBra{f_0}\ldots \widehat{\sBra{f_{i_0}}}\ldots \widehat{\sBra{f_{i_k}}} \ldots\sBra{f_t} \]
for $f_i \in \Gm(U)$ such that $\sum_{i=0}^t f_i =0 $,

\item[$4^\prime$.] (anti-commutativity) $[f][g]-\epsilon[g][f]$,
\end{itemize}
where $\epsilon:=-\langle -1\rangle=-1-\eta[-1]$. 

\begin{lemma}\label{equiv}
The two groups of relations are equivalent in $\HH(U)$.  
\end{lemma}

\begin{proof}
We first assume that 3. and 4. are satisfied. Since 1. and 2. are satisfied, we have $[-f]=[-1]+\langle -1\rangle[f]$. As 4. is satisfied and $[-1]=\epsilon[-1]$ in $\KMW_*(K)$, 
\[
 [-f][f]=[-1][f]+\langle -1\rangle[f]^2= \epsilon([-1][f]-[f]^2) =0
\]
and then $[fg][-fg]=[f][g]+\epsilon [g][f]$ for any $g,f\in \Gm(U)$ by \cite[proof of Lemma 3.7]{Morel08}. Suppose next that $\sum_{i=0}^t f_i =0$, so that
	$ \sum^t_{i=1} \frac{f_i}{-f_0}=1$. Combining 3. and the anti-commutativity law, we obtain
	\begin{align}
		0 &=[1]=\sBra{f_j\inv}+\aBra{f_j\inv}\sBra{f_j} &\text{(by 2.)} \\
		\sBra{\frac{-f_i}{f_j}} &
		=\aBra{f_j\inv} \sBra{-f_i}+\sBra{f_j\inv} & \\
		& =\aBra{f_j\inv}(\sBra{-f_i}-\sBra{f_j}) & \text{(by (4))} \notag \\
		&=\aBra{f_j\inv}(\aBra{-1}\sBra{f_i}+\sBra{-1}-\sBra{f_j})& \notag \\
		(\sBra{f_0}-\sBra{-1})^k &= \sum_{i=0}^{k}\binom{k}{i}\sBra{-1}^{k-i}\sBra{f_0}^i &\\
		&=(\sum_{i=0}^{k-1}\binom{k}{i})\sBra{-1}^{k-1}\sBra{f_0}+(-1)^k\sBra{-1}^k & \text{(by 4.)} \notag\\
		&= (-1)^{k-1}\sBra{-1}^{k-1}\sBra{f_0}+(-1)^k\sBra{-1}^k & \notag
	\end{align}
	 
	\begin{align*}
	0&=(-\aBra{f_0})^t\sBra{\frac{-f_1}{f_0}}\sBra{\frac{-f_2}{f_0}}\ldots\sBra{\frac{-f_t}{f_0}} & \text{(by 3.)} \\
	&=(\sBra{f_0}-\sBra{-1}-\aBra{-1}\sBra{f_1})\ldots(\sBra{f_0}-\sBra{-1} -\aBra{-1}\sBra{f_t})& \text{(by (5))}\\
	&=\epsilon^t \widehat{\sBra{f_0}} \sBra{f_1}\ldots\sBra{f_t}
	+  \sum_{i=1}^t \epsilon^{t-1}\sBra{f_1}\ldots \widehat{\sBra{f_i}}(\sBra{f_0}-\sBra{-1}) \ldots\sBra{f_t} + &\\ 
	&\sum_{i<j}(\sBra{f_0}-\sBra{-1})^2 \sBra{f_1}\ldots \widehat{\sBra{f_i}}\ldots \widehat{\sBra{f_j}} \ldots\sBra{f_t}
	+\ldots &\\
	&=\sum_{i=0}^t \epsilon^{t+i}\sBra{f_0}\ldots \widehat{\sBra{f_i}} \ldots\sBra{f_t} + &\\
	&\sum_{0\leq i_0<\ldots<i_k\leq t} (-1)^k\sBra{-1}^k\sBra{f_0}\ldots \widehat{\sBra{f_{i_0}}}\ldots \widehat{\sBra{f_{i_k}}} \ldots\sBra{f_t} & \text{(by (6))} \\
	&=R(f_0,\ldots,f_t).&
	\end{align*}
Conversely, suppose that $3^\prime$ and $4^\prime$ hold. A direct calculation shows that we have $R(-1,f_1,\ldots,f_t)=(-\aBra{-1})^t \sBra{f_1}\ldots\sBra{f_t}=\epsilon^t \sBra{f_1}\ldots\sBra{f_t}$, and consequently that 3. also holds. For every field $K\neq \bb{F}_2$, we have $1+a+b=0 $ for some $a,b\neq 0$ and it follows from $\sBra{-a}\sBra{-b} =0$ in $\KMW_*(K)$ that 
	\begin{eqnarray*}
	&R(f,af,bf)=R(f,af,bf)-\sBra{-a}\sBra{-b}=R(f,af,bf)-\sBra{-\frac{af}{f}}\sBra{-\frac{bf}{f}}\\
	&=R(f,af,bf)-(\aBra{-1}\sBra{af}+\sBra{-1}-\sBra{f})(\aBra{-1}\sBra{bf}+\sBra{-1}-\sBra{f})\\
	&=-\sBra{-1}\sBra{f}+\sBra{-1}^2-(\sBra{f}-\sBra{-1})^2=\sBra{-1}\sBra{f}-\sBra{f}^2. 
	\end{eqnarray*}
\end{proof}

\begin{remark}
	\label{Relation}
	Observe that the following properties of the relations $R$ and anti-commutativity hold:
\begin{enumerate}
\item For any $a,b\in \Gm(U)$, we have $[a/b]=-\aBra{b\inv}R(b,-a)$.
\item For any $f_0,\ldots,f_t\in\Gm(U)$, by direct computation, we have 
\[
R(f_0,\ldots, f_t) -\epsilon^i\sBra{f_i}R(f_0,\ldots,\widehat{f_i},\ldots,f_t) = P(f_0,\ldots,\widehat{f_i},\ldots,f_t)
\] 
for some polynomial $P$. We shall use the anti-commutativity and the fact that $\sBra{-1}=\epsilon^j\sBra{-1}$ for any $j\geq 0$ in the computation.
\item For any $f_0,\ldots,f_t\in K^\times$ such that $\sum_{i=0}^t f_i =0 $, we have $R(f_0,\ldots,f_t)=0$ in $\KMW_*(K)$.
\end{enumerate}
\end{remark}

The following lemma will prove useful in the proof of the main theorem.

\begin{lemma}\label{lem:basefield}
Any morphism $\phi: \Mt(U^V_I)\to T $ in $\DMt$ such that 
	\[ \Mt(U^Y_{I_Y})(1)[1]\xr{\beta^Y}\Mt(U^V_I)\xr{\phi}T \]
is trivial for every $ Y\in I $ factors through $ \Mt(K)$, i.e. there is a morphism $\psi:\Mt(K)\xr{}T$ such that the following diagram 
	\[\xymatrix{
		\Mt(U^V_I)\ar[d]\ar[r]^-{\phi}&T\\
		\Mt(K)\ar[ur]_-{\psi}&
		}\]
	is commutative.
\end{lemma}

\begin{proof}
We prove as usual the result by induction on $\vert I\vert$, the result being trivial if $\vert I\vert=0$, i.e. if $U_I^V\cong \af^N$. By assumption, $\phi$ factors through $\Mt(U^V_{I-\{Y\}})$, i.e. we have a commutative diagram
\[
\xymatrix{\Mt(U^V_I)\ar[r]^-{\alpha^Y}\ar[rd]_-\phi &  \Mt(U^V_{I-\{Y\}})\ar[d]^-{\phi_0}\\
& T.}
\]
For $H\in I'=I-\{Y\}$, we have an associated Gysin morphism $\beta^H:\Mt(U^H_{I_H})(1)[1]\to \Mt(U^V_I)$ which induces a commutative diagram
	\[\xymatrix{
		\Mt(U^H_{I_H})(1)[1]\ar[d]_-{\alpha^Y(1)[1]}\ar[r]^-{\beta^H}&\Mt(U^V_I)\ar[d]^-{\alpha^Y}\ar[r]^-{\phi}& T\\
		\Mt(U^H_{I'_H})(1)[1]\ar[r]_-{\beta^H}&\Mt(U^V_{I'})\ar[ur]_{\phi_0}&
	}
	\]
	in which the morphism $\alpha^Y(1)[1]$ on the left is split surjective. It follows that $ \phi_0\circ \beta^H \circ \alpha^Y(1)[1] = \phi \circ \beta^H=0 $ implies $ \phi_0\circ \beta^H =0 $. We conclude by induction.
\end{proof}

\begin{proposition}
Let $ S $ be an essentially smooth K-scheme and let $f_i \in \Gm(S)$ be such that $\sum_{i=0}^t f_i =0 $. Then 
\[
R(f_0,\ldots,f_t) = 0\ \text{in}\ \HH(S).
\]
\end{proposition}

\begin{proof}
The global sections $f_0,\ldots,f_t$ yield a morphism $j=(f_0,\ldots,f_t) : S\xr{}\af^{t+1} $ which restrict to a morphism $j:S\xr{}U^H_{I} $, where $ H\subset \af^{t+1}$ is given by $\sum_{i=0}^t x_i =0 $ and $ I=\{\{x_1=0\}, \ldots,\{x_t=0\} \} $. Since $R(f_0,\ldots,f_t)= j^*(R(x_0,\ldots,x_t))$, we can reduce the proposition to $S=U^H_{I}$.

For any $x_j$, we set $Y_j:=\{x_j=0\}\subset H$ and we obtain a Gysin morphism $\beta_j:\Mt(U^{Y_j}_{I_{Y_j}})(1)[1]\to \Mt(U^H_I)$ and a composite
	\[ 
\Mt(U^{Y_j}_{I_{Y_j}})(1)[1]\xr{\beta_j}\Mt(U^H_I)\xr{R(x_0,\ldots,x_t)}\one(t)[t]. 
\] 
By Remark \ref{Relation} and Lemma \ref{lem:computation},
\begin{eqnarray*}
&R(x_0,\ldots,x_t)\circ \beta_j = (\epsilon^j\sBra{x_j}R(x_0,\ldots,\widehat{x_j},\ldots,x_t) + P(x_0,\ldots,\widehat{x_j},\ldots,x_t))\circ  \beta_j\\
		&= \epsilon^j(\sBra{x_j}R(x_0,\ldots,\widehat{x_j},\ldots,x_t))\circ  \beta_j + P(x_0,\ldots,\widehat{x_j},\ldots,x_t)\circ \alpha_j\circ \beta_j\\
		&=  \epsilon^jR(x_0|_{U^{Y_j}_{I_{Y_j}}},\ldots,\widehat{x_j},\ldots,x_t|_{U^{Y_j}_{I_{Y_j}}}).
	\end{eqnarray*}
	
As $R(f,-f)=0$ for $f\in \Gm(S)$ by Remark \ref{Relation}, we obtain by induction that $R(x_0,\ldots,x_t)\circ \beta_j =0$ for any $j=0,\ldots,t$. Applying Lemma \ref{lem:basefield}, we obtain a commutative diagram 
\[
\xymatrix{
		\Mt(U^H_{I})\ar[d]\ar[rr]^{R(x_0,\ldots,x_t)}&&\one(t)[t]\\
		\Mt(K)\ar[urr]_{\psi}&
	}
\]
	As $\mathrm{char}(K)\neq 2$, $ U^H_{I}$ has a $K$ rational point $ (\lambda_0,\ldots,\lambda_t)\in \af^{t+1}$, and we obtain a diagram
	\[\xymatrix{
		\Mt(K)\ar[d]_u\ar[drr]^{R(\lambda_0,\ldots,\lambda_t)}&\\
		\Mt(U^H_{I})\ar[d]\ar[rr]^{R(x_0,\ldots,x_t)}&&\one(t)[t].\\
		\Mt(K)\ar[urr]_{\psi}&
	}\]
	The vertical composite being the identity, $\psi =R(\lambda_0,\ldots,\lambda_t)$ and the latter is trivial by the relations in Milnor-Witt $K$-theory.
\end{proof}

Applying lemma \ref{equiv}, we obtain the following corollary.

\begin{corollary}\label{cor:relations}
	Let $ S $ be an essentially smooth smooth $ K $-scheme.
	\begin{enumerate}
		\item For any $f_1,\ldots, f_t\in  \Gm(S)$ such that $\sum_{i=1}^t f_i =1 $, we have  
		\[
		[f_1][f_2]\cdots[f_t]=0\in \HH(S).
		\]
		\item For any $ f \in\Gm(S)$, we have $[f]^2-[-1][f]=0$ in $\HH(S)$. 
	\end{enumerate}
\end{corollary}

\section{Proof of the main theorem}
In this section, we prove Theorem \ref{thm:mainthm}. We denote by $ J_U\subset \HH(K) \{\Gm(U)\}$ the ideal generated by the relations 
\begin{enumerate}[label=(\arabic*)]
	\item $(f)-[f]$, for $ f \in K^{\times}\subset \Gm(U) $.
	\item $ (f)+(g)+\eta(f)(g)-(fg) $, for $ f,g\in \Gm(U) $.
	\item $(f_1)(f_2)\cdots(f_t)$, for any $f_1,\ldots,f_t \in \Gm(U)$ such that $\sum_{i=1}^t f_i =1$.
	\item $(f)^2-[-1](f)$, for $ f \in\Gm(U)$. 
\end{enumerate}
By Lemma \ref{equiv}, $J_U\subset \HH(K) \{\Gm(U)\}$ is in fact generated by 
\begin{itemize}
	\item[(1)] $(f)-[f]$, for $ f \in K^{\times}\subset \Gm(U) $.
	\item[(2)] $ (f)+(g)+\eta(f)(g)-(fg) $, for $ f,g\in \Gm(U) $.
	\item[($3^\prime$)] (anti-commutativity) $(f)(g)-\epsilon(g)(f)$ for any $f,g\in \Gm(U)$.
	\item[($4^\prime$)] $R(f_0,\ldots,f_t)$ given by 
	\[
	\sum_{i=0}^t \epsilon^{t+i}(f_0)\ldots \widehat{(f_i)} \ldots(f_t) +
	\sum_{0\leq i_0<\ldots<i_k\leq t} (-1)^k\sBra{-1}^k(f_0)\ldots \widehat{(f_{i_0})}\ldots \widehat{(f_{i_k})} \ldots(f_t),
	\] 
	for any $f_0,\ldots,f_t \in \Gm(U)$ such that $\sum_{i=0}^t f_i =0$.
\end{itemize}

In view of Corollary \ref{cor:relations}, the morphism $\HH(K)\{\Gm(U)\}\to \HH(U)$ defined by $(f)\mapsto [f]$ induces a morphism of $\HH(K)$-algebras 
\[
\rho:\\ \HH(K)\{\Gm(U)\}/J_U \to \HH(U).
\] 
Now, choose linear polynomials $\phi_1,\ldots,\phi_s$ that define the hyperplanes $Y_i\in I$ and let $J^\prime_U\subset \HH(K)\{\Gm(U)\}$ be the ideal generated by the relations (1), (2), ($3^\prime$) and ($4^\prime$) for elements of the form $f_j=\lambda_j\phi_{i_j}$ or $f_j=\lambda_j$ for $\lambda_j\in K^\times$ and $\phi_{i_j}\in \{\phi_1,\ldots,\phi_s\}$. We have a string of surjective morphisms of $\HH(K)$-algebras
\[
\HH(K)\{\Gm(U)\}/J^\prime_U\to\HH(K)\{\Gm(U)\}/J_U \xr{\rho} \HH(U)
\]
whose composite we denote by $\rho^\prime$. 

\begin{theorem}\label{thm:main}
The morphism of $\HH(K) $ algebra 
\[
\HH(K)\{\Gm(U)\}/J_U \xr{\rho} \HH(U)
\] is an isomorphism.
\end{theorem}

\begin{proof}
It suffices to prove that $\rho^\prime$ is an isomorphism. To see this, we work again by induction on $\vert I\vert$. If $\vert I\vert=0$, we have $U\cong \af^N$ for some $N\in\bb{N}$. By homotopy invariance, we have to prove that the map
\[
\rho^\prime:\HH(K)\{\Gm(K)\}/J^\prime_K \to \HH(K)
\]
is an isomorphism. Now, the morphism of $\HH(K)$-algebras
\[\HH(K)\to \HH(K)\{\Gm(K)\}/J^\prime_K
\] 
is surjective by Relation 1. Its composite with $\rho^\prime$ is the identity and we conclude in that case.

Assume now that $Y\in I$ is defined by $\phi_1=0$ and that we have isomorphisms
	\[ \HH(K)\{\Gm(U^V_{I'})\}/J^\prime_{U^V_{I'}} \xr{\sim} \HH(U^V_{I'}),\]
	\[\HH(K)\{\Gm(U^Y_{I_Y})\}/J^\prime_{U^Y_{I_Y}} \xr{\sim} \HH(U^Y_{I_Y}).\]
The morphism $U_I^V\to U^V_{I'}$ induces a morphism  $\Gm(U^V_{I'})\to \Gm(U^V_{I})$ and then a commutative diagram
\[\small
\xymatrix@C=1.0em{&  \HH(K)\{\Gm(U^V_{I'})\}/J^\prime_{U^V_{I'}}\ar[r]^-{\tilde\alpha}\ar[d]_-{\cong} & \HH(K)\{\Gm(U_I^V)\}/J^\prime_{U^V_I}\ar[d]^-{\rho^\prime}\ar[r]^-{\tilde\beta} & \HH(K)\{\Gm(U^Y_{I_Y})\}/J^\prime_{U^Y_{I_Y}}\ar[d]^-{\cong} & \\
0\ar[r] & \HH(U^V_{I'}) \ar[r]_-{\alpha^Y_*} & \HH(U^V_{I})\ar[r]_-{\beta^Y_*} & \HH(U^Y_{I_Y})\ar[r] & 0}
\]
in which $\tilde\beta$ is the unique lift of $	\beta^Y_*\circ\rho$ and the bottom row is exact. We are thus reduced to prove that the top sequence is short exact to conclude. It is straightforward to check that $ \tilde{\alpha} $ is injective and $ \tilde{\beta} $ is surjective. Moreover, the commutativity of the diagram and the fact that $\beta^Y_*\circ \alpha^Y_*=0$ imply that $\tilde\beta\circ\tilde\alpha=0$, so we are left to prove exactness in the middle. 

Let $x\in  \HH(K)\{\Gm(U_I^V)\}/J^\prime_{U^V_I}$. The group $\Gm(U_I^V)$ being generated by $\Gm(U_{I^\prime}^V)$ and $\phi_1$, we may use relations (2) and (4) to see that $x=(\phi_1)\tilde{\alpha}(x_1)+\tilde{\alpha}(x_0)$ in $\HH(K)\{\Gm(U_I^V)\}/J^\prime_{U^V_I}$. By Lemma \ref{lem:computation}, we get $\tilde{\beta}(x)=\tilde{\iota}(x_1)$, where $ \tilde{\iota} $ is induced by the restriction $ \Gm(U^V_{I'}) \xr{}\Gm(U^Y_{I_Y})$. Consequently, we need to prove that if $\tilde{\iota}(x_1)=0$ then $(\phi_1)\tilde{\alpha}(x_1)$ is in the image of $\tilde\alpha$. With this in mind, we now prove that the kernel of $\tilde{\iota}$ is generated by elements of the form
\[
R(f_0,\ldots,f_t)
\]
where $f_j=\lambda \phi_{i_j}$ with $i_j>1$ or $f_j=\lambda$ and $\sum_{i=0}^t f_i|_{U^Y_{I_Y}} =0$. Denote by $L^\prime$ the ideal of $ \HH(K)\{\Gm(U^V_{I'})\}$ generated by such elements. By construction, the restriction induces a homomorphism
\[
L^\prime+J^\prime_{U^V_{I'}}\to J^\prime_{U^Y_{I_Y}}
\]
which is surjective. Indeed, relations (1), (2) and ($3^\prime$) can be lifted using the fact that the map $\Gm(U^V_{I'})\to \Gm(U^Y_{I_Y})$ is surjective, while an element satisfying relation (4) with every $f_j$ of the form $f_j=\lambda_j\phi_{i_j}$ or $f_j=\lambda_j$ for $\lambda_j\in K^\times$ (with $i_j\neq 1$) lifts to an element in $L^\prime$. As in \cite[proof of Theorem 3.5]{math/0601737}, we see that the kernel of the group homomorphism $\Gm(U^V_{I'})\to \Gm(U^Y_{I_Y})$ is generated by elements of the form
\begin{enumerate}
\item $\lambda \frac {\phi_i}{\phi_j}$ with $i,j$ such that $Y_1\cap Y_i=Y_1\cap Y_j$ and $\lambda=\frac{(\phi_j)_{\vert Y_1}}{(\phi_i)_{\vert Y_1}}$.
\item $\lambda\phi_i$ where $i$ is such that $Y_1\cap Y_i=\emptyset$ and $\lambda=\frac 1{(\phi_i)_{\vert Y_1}}$.
\end{enumerate}
Remark \ref{Relation} yields $[\frac {\lambda\cdot \phi_i}{\phi_j}]=-\aBra{\phi_j\inv}R(\phi_j,-\lambda\cdot \phi_i)\subset L^\prime$, while $[\lambda\phi_i]=\epsilon R(-1,\lambda\cdot \phi_i)\subset L^\prime$ showing that $\ker (\Gm(U^V_{I'})\to \Gm(U^Y_{I_Y}))\subset L^\prime+J^\prime_{U^V_{I'}}$. We deduce that $\ker(\tilde{\iota})=L^\prime$.

We now conclude. If $\tilde{\iota}(x_1)=0$, then $x_1\in L^\prime$ and we may suppose that $x_1=R(f_0,\ldots,f_t)$ for $f_0,\ldots,f_t$ such that $\sum_{i=0}^t f_i|_{U^Y_{I_Y}} = 0 $. It follows that $\sum_{i=0}^t f_i = -\mu \phi_1$ for  $\mu \in K$. If $\mu=0$ there is nothing to do. Else, use $ R(\mu \phi_1,f_0,\ldots,f_t)=0 $ and Remark \ref{Relation} to get
	\begin{eqnarray*}
		&(\phi_1)\tilde{\alpha}(x_1)&= (\mu\phi_1)\tilde{\alpha}(x_1)-\aBra{\phi_1}(\mu)\tilde{\alpha}(x_1)\\
		&&= (\mu\phi_1)\tilde{\alpha}(x_1)+R(\mu \phi_1,f_0,\ldots,f_t)-\aBra{\phi_1}(\mu)\tilde{\alpha}(x_1)\\
		&&=\tilde{\alpha}(P(f_0,\ldots,f_t)) -\tilde{\alpha}(\aBra{\phi_1}(\mu)x_1) \in \mrm{image}(\tilde{\alpha}).
	\end{eqnarray*}
\end{proof}

\begin{corollary}
The graded ring isomorphism of Theorem \ref{thm:main} induces an isomorphism 
\[
\bigoplus_{n\in\bb{Z}}\mathrm{K}_n^{\rMW}(K)\{\Gm(U)\}/J_U\to \bigoplus_{n\in\bb{Z}}\mathrm{H}_{\rMW}^{n,n}(U,\bb{Z}).
\]
\end{corollary}

\begin{proof}

Notice that the ideal $J_U$ of Theorem \ref{thm:main} is homogeneous, and it follows that $\bigoplus_{n\in\bb{Z}}\mathrm{H}_{\rMW}^{n,n}(U,\bb{Z})$ can be computed as $\mathrm{H}_{\rMW}^{*,*}(K)\{\Gm(U)\}/J_U$, where $\mathrm{H}_{\rMW}^{*,*}(K)$ is the diagonal of $\HH(K)$.

\end{proof}
	
\section{Combinatorial description}

In this section, we fix an affine space $V=\af^N$, a family of hyperplanes $I$ and we set $U:=U_I^N$. We let $Q(U)$ be the cokernel of the group homomorphism $ \Gm(K)\to \Gm(U)$, and we observe that the divisor map 
\[
\Gm(U)\xr{\mrm{div}} \oplus_{Y_i \in I} \bb{Z} \cdot Y_i
\] 
in $ \af^N $ induces an isomorphism $ Q(U)\cong \oplus_{Y_i \in I} \bb{Z} \cdot Y_i$. We consider the exterior algebra $\Lambda_{\bb{Z}} Q(U)$ and write $\Lambda_{\bb{Z}[\eta]/2\eta}Q(U):=\bb{Z}[\eta]/2\eta \otimes_{\bb{Z}}\Lambda_{\bb{Z}}Q(U)$. The abelian group $Q(U)$ being free, the $\bb{Z}[\eta]/2\eta$-module $\Lambda_{\bb{Z}[\eta]/2\eta}Q(U)$ is also free, with usual basis. To provide a combinatorial description of $ \HH(U)$, we will have to slightly modify the definition of the divisor map above, in order to incorporate the action of $\eta$. We then define a map 
\[
\Gm(U)\xr{\widetilde{\mrm{div}}} \Lambda_{\bb{Z}[\eta]/2\eta}Q(U)
\]
as follows:
\begin{enumerate}
	\item If $f=\lambda\phi$ or $f=\lambda$ where $\lambda \in \Gm(K)$ and $\phi$ is a linear polynomial as above, then $ \widetilde{\mrm{div}}(f) = \mrm{div}(f) $.
	\item If $f,g\in \Gm(U)$ then $\widetilde{\mrm{div}}(fg)= \widetilde{\mrm{div}}(f) + \widetilde{\mrm{div}}(g) + \eta \cdot \widetilde{\mrm{div}}(f) \wedge \widetilde{\mrm{div}}(g)$.
\end{enumerate}

\begin{lemma}\label{lem:tildediv}
The map $\widetilde{\mrm{div}}$ is well-defined.
\end{lemma}

\begin{proof}
We first notice that $\widetilde{\mrm{div}}(fg)=\widetilde{\mrm{div}}(gf)$, since
\[\widetilde{\mrm{div}}(fg)-\widetilde{\mrm{div}}(gf)=\eta \cdot \widetilde{\mrm{div}}(f) \wedge \widetilde{\mrm{div}}(g)-\eta \cdot \widetilde{\mrm{div}}(g) \wedge \widetilde{\mrm{div}}(f)=2\eta \cdot \widetilde{\mrm{div}}(f) \wedge \widetilde{\mrm{div}}(g) = 0.\]
	
Let $f_1,f_2,g_1,g_2\in \Gm(U)$ be such that $f_1g_1=f_2g_2$. Let $Y\in I$ be such that $f_i=Y^{n_i}\cdot f_i^\prime$ with $\mrm{div}_Y(f_i^\prime)=0$ and $g_i=Y^{m_i}\cdot g_i^\prime$ with $\mrm{div}_Y(g_i^\prime)=0$ for $i=1,2$ and $m_i,n_i\in\bb{Z}$. We get 
\[
\widetilde{\mrm{div}}(f_1g_1)=(m_1+n_1)\cdot Y+\widetilde{\mrm{div}}(f_1^\prime g_1^\prime)+(m_1+n_1)\eta (Y\wedge \widetilde{\mrm{div}}(f_1^\prime g_1^\prime))
\]
and 
\[
\widetilde{\mrm{div}}(f_2g_2)=(m_2+n_2)\cdot Y+\widetilde{\mrm{div}}(f_2^\prime g_2^\prime)+(m_2+n_2)\eta (Y\wedge \widetilde{\mrm{div}}(f_2^\prime g_2^\prime))
\]
As $\Lambda_{\bb{Z}[\eta]/2\eta}Q(U)$ is free with usual basis, we deduce that $\widetilde{\mrm{div}}(f_2^\prime g_2^\prime)=\widetilde{\mrm{div}}(f_1^\prime g_1^\prime)$ which allows to conclude by induction on the number of non-trivial factors in the decomposition of $f_1g_1$.
\end{proof}

Let now $L_U\subset \Lambda_{\bb{Z}[\eta]/2\eta}Q(U)$ be the ideal generated by the following elements:
\begin{enumerate}
\item $Y_1\wedge\ldots\wedge Y_s $, for $Y_i\in I$ such that $ Y_1\cap\ldots\cap Y_s =\emptyset $;
\item $\sum^s_{j=1}  (-1)^kY_1\wedge\ldots\wedge \widehat{Y_j}\wedge\ldots\wedge Y_s $, for $Y_i\in I$ such that  $ Y_1\cap\ldots\cap Y_s \neq \emptyset $ and $ \mrm{codim}(Y_1\cap\ldots\cap Y_s)<s $.
\end{enumerate}

As a consequence of Lemma \ref{lem:tildediv}, the map $\widetilde{\mrm{div}}$ induces a morphism of $\bb{Z}[\eta]/2\eta$-algebras
\[
\psi:(\bb{Z}[\eta]/2\eta)\{\Gm(U)\}\to \Lambda_{\bb{Z}[\eta]/2\eta}Q(U)/L_U
\]

It is now time to introduce the ring 
\[
A_0(U) := \mathrm{K}_*^{\rMW}(K)\{\Gm(U)\}/\left(J_U + K^{\times} \cdot \mathrm{K}_*^{\rMW}(K)\{\Gm(U)\} \right).
  \]
As $\epsilon=-1-[-1]\eta \sim -1 $ in $A_0(U)$, it follows that $A_0(U)$ is an exterior algebra. Moreover, the coefficient ring $\mathrm{K}_*^{\rMW}(K)$ can be reduced to $\mathrm{K}_*^{\rMW}(K)/(K^{\times} \cdot \mathrm{K}_*^{\rMW}(K) ) \cong \bb{Z}[\eta]/2\eta$.

\begin{proposition}
The morphism of $\bb{Z}[\eta]/2\eta$-algebras
\[
\psi:\bb{Z}[\eta]/2\eta\{\Gm(U)\}\to \Lambda_{\bb{Z}[\eta]/2\eta}Q(U)/L_U
\]
induces an isomorphism
\[
\Psi : A_0(U) \to \Lambda_{\bb{Z}[\eta]/2\eta}Q(U)/L_U.
\]
\end{proposition}

\begin{proof}
We first prove that $\Psi$ is well-defined, which amounts to show that the image of $J_U$ is contained in $L_U$. For $f\in K^\times$, we have $[f]\in K^\times \cdot \mathrm{K}_*^{\rMW}(K)\{\Gm(U)\}$ and $\widetilde{\mrm{div}}(f)=0$, showing that the first relation is satisfied. The second relation is satisfied by definition of $\widetilde{\mrm{div}}$, while relation ($3^\prime$) is satisfied as $ \Lambda_{\bb{Z}[\eta]/2\eta}Q(U)/L_U$ is an exterior algebra. As in the proof of Theorem \ref{thm:main}, we are then left with elements of $J^\prime_U$, i.e. elements of the form $R(f_0,\ldots,f_t)$ for $\sum_{i=0}^t f_i =0$, where $f_j=\lambda_j\phi_{j}$ or $f_j=\lambda_j$. Modulo $K^\times \cdot \mathrm{K}_*^{\rMW}(K)\{\Gm(U)\}$, we have $ R(f_0,\ldots,f_t) \sim \sum_{i=0}^t (-1)^{t+i}\sBra{f_0}\ldots \widehat{\sBra{f_i}} \ldots\sBra{f_t} $ and we just need to prove that
\[
\alpha:=(-1)^t\psi(R(f_0,\ldots,f_t))= \sum_{i=0}^t (-1)^{i}\widetilde{\mrm{div}} (f_0)\wedge\ldots \wedge\widehat{\widetilde{\mrm{div}}(f_i)}\wedge \ldots\wedge\widetilde{\mrm{div}}(f_t)
\]
is an element of $L_U$. Note that if there are more than two constant functions among $f_j$, $\alpha$ would be trivial. Suppose that $f_0=\lambda_0$ is the only constant, and let $f_j=\lambda_j \phi_j$ with kernel $Y_j\in I$, so that $\alpha =Y_1\wedge\ldots\wedge Y_t$. Since $\sum_{j=1}^t\lambda_j \phi_j=-\lambda_0\neq0$, we can easily get that $Y_1\cap\ldots\cap Y_t =\emptyset$ and $\alpha =Y_1\wedge\ldots\wedge Y_t \in L_U$. In the case where none of the $f_j$ are constant, $\alpha = \sum_{i=0}^t (-1)^{i}Y_0\wedge\ldots \wedge\widehat{Y_i}\wedge \ldots\wedge Y_t$. And for every $i$, we have$ \sum_{j=0,j\neq i}^t\lambda_j \phi_j=-\lambda_i\phi_i$, which means $Y_i\subseteq Y_0\cap\ldots\cap \widehat{Y_i} \cap\ldots\cap Y_t=Y_0\cap\ldots\cap Y_t$. If $Y_0\cap\ldots\cap Y_t = \emptyset $, so is $Y_0\cap\ldots\cap \widehat{Y_i} \cap\ldots\cap Y_t$, thus $Y_0\wedge\ldots \wedge\widehat{Y_i}\wedge \ldots\wedge Y_t \in L_U $; else, $\mrm{codim}(Y_0\cap\ldots\cap Y_t)=\mrm{codim}(Y_0\cap\ldots\cap \widehat{Y_i} \cap\ldots\cap Y_t)\leq t < t+1$, which just fits the condition (2) of $L_U$. This proves that $\Psi$ is well-defined.
	
To prove that $\Psi$ is an isomorphism, we construct the inverse map by
\[
\Phi : \Lambda_{\bb{Z}[\eta]/2\eta}Q(U)/L_U\to A_0(U), Y_i \mapsto (\phi_i) 
\] 
and prove that it is well-defined. As above, we just need to discuss elements of $L_U$. If $ Y_1\cap\ldots\cap Y_s =\emptyset $, then we can find $\lambda_i\in K^\times$ such that $\sum_i \lambda_i \phi_i =1$, and thus $(\phi_1)\cdots(\phi_s)\sim(\lambda_1\phi_1)\cdots(\lambda_s\phi_s)= 0 $ in $A_0(U)$. In the case $\mrm{codim}(Y_1\cap\ldots\cap Y_s)<s$, we have $\sum_i \lambda_i \phi_i =0$ for some $\lambda_i\in K^\times$. Then $\sum_{i=1}^s (-1)^i(\phi_1)\ldots \widehat{(\phi_i)} \ldots(\phi_s)\sim (-1)^{s-1}R(\lambda_1 \phi_1,\ldots,\lambda_s \phi_s) = 0$ in  $A_0(U)$. This shows that the inverse map is well-defined.
\end{proof}

The following corollary shows that the rank of the free $\HH(K)$-module $\HH(U)$ is exactly the same as the rank of the free $\mrm{H}_{\mathrm{M}}(K)$-module $\mrm{H}_{\mrm{M}}(U)$ \cite[Proposition 3.11]{math/0601737}.

\begin{corollary}
The rank of the free $\HH(K)$-module $\HH(U)$ is equal to the rank of the free module $\Lambda_{\bb{Z}}Q(U)/L_U $.
\end{corollary}

\begin{proof}
It is clear that $\mrm{rk}_{\bb{Z}[\eta]/2\eta}(\Lambda_{\bb{Z}[\eta]/2\eta}Q(U)/L_U)=\mrm{rk}_{\bb{Z}}(\Lambda_{\bb{Z}}Q(U)/L_U)$. As all generators in $\HH(U)$ are from $\HH^{p,p}(U,\bb{Z})$, we have
\[
\mrm{rk}_{\HH(K)}(\HH(U))=\mrm{rk}_{\KMW_*(K)}(\bigoplus_{n\in\bb{Z}}\HH^{n,n}(U,\bb{Z}))=\mrm{rk}_{\bb{Z}[\eta]/2\eta}(A_0(U))
\]
\end{proof}

\section{\textbf{I}-cohomology and singular cohomology}
In ordinary motivic cohomology theory, we have a realization functor to the topological cohomology of complex points. This yields the following comparative result.
\begin{proposition}{\cite[Proposition 3.9]{math/0601737}}
	In the case $K=\bb{C}$, there is an isomorphism of rings:
	\[ 
	\oplus_n \mrm{H}_{\mrm{M}}^{n,n}(U,\bb{Q}) \otimes_{\mrm{H}_{\mrm{M}}(K)} \mrm{K}^\mrm{M}_*(K)/ K^{\times}\cdot \mrm{K}^\mrm{M}_*(K) \xr{\cong} \oplus_n \mrm{H}^{n}_{\mrm{sing}}(U(\bb{C}),\bb{Q}). \]
\end{proposition}

In this section, we provide an analogue for the singular cohomology of the real points of the complement of a hyperplane arrangement defined over $\mathbb{R}$. We start with some results about the \textbf{I}-cohomology \cite{Fasel08a}.

As recalled in Section \ref{sec:MWmotivic}, we have natural homomorphisms from \MW motivic cohomology to $\mathbf{I}^*$-cohomology
\[
\HH ^{p,q}(X,\bZ)\to \mrm{H}^{p-q}(X,\mathbf{K}_q^{\mathrm{MW}})\to \mrm{H}^{p-q}(X,\mathbf{I}^q)
\]
which induce a ring homomorphism $ \HH(X)\to \oplus_{r,q}\mrm{H}^{r}(X,\mathbf{I}^q) $ (where $\mathbf{I}^q = \mathbf{K}_q^{\mathrm{MW}} =\mathbf{W}$ for $q<0$). In case $X=\mrm{Spec}(K)$, we obtain in particular a ring homomorphism $ \HH(K)\to \oplus_{r,q}\mrm{H}^{r}(K,\mathbf{I}^q) = \oplus_{q\in \bZ} \mathrm{I}^q(K)  $

\begin{proposition}
The morphism of $ \oplus_{q\in \bZ} I^q(K) $-algebras
\[
j: \HH(U)\otimes_{\HH(K)} \left(\oplus_{q\in \bZ} \mathrm{I}^q(K)\right) \to\oplus_{r,q}\mrm{H}^{r}(U,\mathbf{I}^q) 
\]
is an isomorphism. Moreover $ \mrm{H}^{r}(U,\mathbf{I}^q) =0 $, for $r\neq 0$.
\end{proposition}

\begin{proof}
Throughout the proof, we write $ \HH(U)\otimes \mathbf{I} $ for the graded ring $ \HH(U)\otimes_{\HH(K)} (\oplus_q \mathrm{I}^q(K))$. We follow the same induction process as in the proof of the main theorem. When $ |I|=0 $, we only need to consider $ \mrm{Spec}(K)$ by homotopy invariance, and the result is trivial.
	
	Assume now that $ Y\in I $ and that we have isomorphisms for $ U^V_{I'}$ and $ U^Y_{I_Y} $.
	Notice that for \textbf{I}-cohomology we still have a Gysin long exact sequence \cite[Remarque 9.3.5]{Fasel08a}. The proof of the main theorem yields the following commutative diagram:
	\[\small
	\xymatrix@C=1.0em{
		0\ar[r] & \HH(U^V_{I'})\otimes \mathbf{I}\ar[r]\ar[d]_-{\cong} & \HH(U^V_{I})\otimes \mathbf{I}\ar[r]\ar[d]_-{j} & \HH(U^Y_{I_Y})\otimes \mathbf{I}\ar[r]\ar[d]_-{\cong} & 0 \\
		\oplus_q\mrm{H}^{-1}(U^Y_{I_Y},\mathbf{I}^{q-1})\ar[r] & \oplus_q\mrm{H}^{0}(U^V_{I'},\mathbf{I}^q)\ar[r] & \oplus_q\mrm{H}^{0}(U^V_{I},\mathbf{I}^q)\ar[r] & \oplus_q\mrm{H}^{0}(U^Y_{I_Y},\mathbf{I}^{q-1})\ar[r] & \oplus_q\mrm{H}^{1}(U^V_{I'},\mathbf{I}^q) 
	 }
	\]
By our assumption, $ \mrm{H}^{-1}(U^Y_{I_Y},\mathbf{I}^{q-1})$ and $ \mrm{H}^{1}(U^V_{I'},\mathbf{I}^q)  $ are both $ 0 $, so the second line is also short exact. We conclude that $j$ is an isomorphism as well. The same argument implies that $ \mrm{H}^{r}(U^V_{I},\mathbf{I}^q)=0 $ for $ r\neq0 $. 
\end{proof}

The analogue of Corollary \ref{cor:free} in this setting then reads as follows.

\begin{corollary}\label{decompI}
There is a finite set $J$ and integers $n_j \geq 0$ for any $j\in J$ such that
\[
\mrm{H}^{0}(U^V_I,\mathbf{I}^q)\cong \oplus_{j\in J} \mrm{I}^{q-n_j}(K) b_j
\]
as a free $ \oplus_q \mrm{I}^q(K)$-module with basis $b_j\in \mrm{H}^{0}(U^V_I,\mathbf{I}^{n_j})$.
\end{corollary}

\begin{proof}
Every step is the same as in Proposition \ref{decomp}, except the splitting, which comes from the identification with $ \HH(U^V_{I})\otimes \mathbf{I} $. 
\end{proof}

As in \cite{Jacobson17} and \cite{hornbostel2019real}, we can compute the cohomology of the real spectrum using \textbf{I}-cohomology.

\begin{proposition}{\cite[Proposition 3.6]{hornbostel2019real}}
The signature map induces an isomorphism 
\[
\mrm{H}^{r}(X,\mrm{Colim}_{q\geq 0}\mathbf{I}^q)\xr{\mrm{sign}_{\infty}} \mrm{H}^{r}_{\mrm{sing}}(\mrm{Sper}(X),\bZ)  
\]
where $ \mrm{Sper}X $ is the real spectrum. In particular, 
\[
\mrm{Colim}_{q\geq 0} \mrm{I}^q(K)\cong\mrm{H}^{0}_{\mrm{sing}}(\mrm{Sper}(K),\bZ).
\]
\end{proposition}

In our case, since $ U $ is always noetherian and $ \mrm{Colim}_{q\geq 0} $ is filtered, we have a canonical isomorphism 
\[
\mrm{H}^{r}(U,\mrm{Colim}_{q\geq 0}\mathbf{I}^q) \cong \mrm{Colim}_{q\geq 0}\mrm{H}^{r}(U,\mathbf{I}^q).
\] 
Combining with Corollary \ref{decompI}, we obtain following proposition.
\begin{proposition}
There exists an integer $N>0$ such that 
\[
\mrm{H}^{0}(U^V_I,\mathbf{I}^N)\otimes_{\oplus_{q\geq 0} I^q(K)}\mrm{H}^{0}_{\mrm{sing}}(\mrm{Sper}(K),\bZ)\xr{2^{-N}\mrm{sign}} \mrm{H}^{0}_{\mrm{sing}}(\mrm{Sper}(U^V_I),\bZ)  
\]
is an isomorphism. Moreover $ \mrm{H}^{r}_{\mrm{sing}}(\mrm{Sper}(U^V_I),\bZ)=0 $ for $r\neq 0$.
\end{proposition}

\begin{proof}
As  discussed above, we can rewrite the right hand side as \\$ \mrm{Colim}_{q\geq 0}\mrm{H}^{0}(U^V_I,\mathbf{I}^q)$. Applying Corollary \ref{decompI}, we get 
\[ 
\mrm{Colim}_{q\geq 0} \left(\oplus_{j\in J} \mrm{I}^{q-n_j}(K) b_j\right) \cong \oplus_{j\in J} \left(\mrm{Colim}_{q\geq 0} \mrm{I}^{q-n_j}(K) b_j\right)\]
\[ \cong \oplus_{j\in J} \mrm{H}^{0}_{\mrm{sing}}(\mrm{Sper}(K),\bZ) b_j .
\]
Let $N\in\mathbb{N}$ be such that $\forall j\in J, N\geq n_j $. Using again Corollary \ref{decompI}, 
\[
\mrm{H}^{0}(U^V_I,\mathbf{I}^N)\cong \oplus_{j\in J} \mrm{I}^{N-n_j}(K) b_j
\] 
which implies  
\[ 
\oplus_{j\in J} \mrm{I}^{N-n_j}(K) b_j \otimes_{\left(\oplus_{q\geq 0} \mrm{I}^q(K)\right)}\mrm{H}^{0}_{\mrm{sing}}(\mrm{Sper}(K),\bZ) \cong \oplus_{j\in J} \mrm{H}^{0}_{\mrm{sing}}(\mrm{Sper}(K),\bZ) b_j
\]
since for every $j$ we have $N-n_j \geq 0 $. That proves the first part, while the second part is trivial.
\end{proof}

Taking $ K =\bb{R}$, we have $\mrm{H}^{0}_{\mrm{sing}}(\bb{R},\bZ)=\bZ$ and we recover the classical result for complements of hyperplane arrangements:
\[\mrm{H}^{0}(U^V_I,\mathbf{I}^N)\xrightarrow[\cong]{2^{-N}\mrm{sign}} \mrm{H}^{0}_{\mrm{sing}}(U^V_I(\bb{R}),\bZ)\cong \bigoplus_{R_i\in  \text{connected components}} \bb{Z}\{R_i\}. \]

\bibliography{references}
\bibliographystyle{unsrt}

\end{document}